\documentclass[12pt]{article}

\usepackage[margin=2.5cm]{geometry}

\usepackage{amsmath, amssymb, latexsym, amsthm}

\usepackage{graphicx}
\usepackage[dvipsnames]{color}

\newtheorem{theorem}{Theorem}[section]
\newtheorem{proposition}[theorem]{Proposition} 

\newtheorem{corollary}[theorem]{Corollary}
\newtheorem{lemma}[theorem]{Lemma}
\newtheorem{conjecture}[theorem]{Conjecture}

\newtheorem*{theorem*}{Theorem}
\newtheorem*{conjecture*}{Conjecture}
\newtheorem*{corollary*}{Corollary}
\newtheorem*{proposition*}{Proposition}

\newtheorem{problem}[theorem]{Problem}

\theoremstyle{definition}

\theoremstyle{remark}

\newcommand{\muint}{\mu_{\mathrm{int}}}

\definecolor{gruen}{cmyk}{1.0,0.2,0.7,0.07}
\definecolor{mag}{cmyk}{0.0,0.9,0.3,0.0}

\begin{document}

\title{Improper interval edge colorings of graphs}

\author{Carl Johan Casselgren, Petros A. Petrosyan}

\author{
{\sl Carl Johan Casselgren}\thanks{{\it E-mail address:} 
carl.johan.casselgren@liu.se. \, Casselgren was supported by a grant from the Swedish
Research Council (2017-05077)}\\ 
Department of Mathematics \\
Link\"oping University \\ 
SE-581 83 Link\"oping, Sweden
\and
{\sl Petros A. Petrosyan}\thanks{{\it E-mail address:} 
petros\_petrosyan@ysu.am} \\
Department of Informatics \\ and Applied Mathematics,\\
Yerevan State University \\ 0025, Armenia
}

\maketitle


\begin{abstract}
A \emph{$k$-improper edge coloring} of a graph $G$ is a mapping 
$\alpha:E(G)\longrightarrow \mathbb{N}$ such that at most $k$ 
edges of $G$ with a common endpoint have the same color. 
An improper edge coloring of a graph $G$ is called an \emph{improper interval edge coloring} if the colors of the 
edges incident to each vertex of $G$ form an integral interval. 
In this paper we introduce and
investigate a new notion, the {\em interval coloring impropriety}
(or just {\em impropriety})
of a graph $G$ defined as the smallest $k$ such that $G$ has a $k$-improper interval edge coloring; 
we denote the smallest such $k$ by $\muint(G)$.
We prove upper bounds on $\muint(G)$ for 
general graphs $G$ and for particular families
such as bipartite, complete multipartite and outerplanar  
graphs;
we also determine $\muint(G)$ exactly for $G$ belonging to
some particular classes of graphs.
Furthermore,
we provide several families of graphs with large impropriety; in particular, we prove that for each positive integer $k$, 
there exists a graph $G$ with $\muint(G) =k$.
Finally, for graphs with at least two vertices we prove a
new upper bound on the number of colors used in an improper interval edge coloring.
\end{abstract}

	\section{Introduction}

A proper $t$-edge coloring of a graph $G$ is called an
\emph{interval $t$-coloring} if the colors of the edges incident to every
vertex $v$ of $G$ form an interval of integers.
 This notion was introduced by
Asratian and Kamalian \cite{AsrKam}
(available in English as \cite{AsrKamJCTB}),
motivated by the problem of constructing timetables
 without ``gaps'' for teachers and classes.
Generally, it is an NP-complete problem
to determine whether a bipartite graph
has an interval coloring \cite{Seva}. However some classes of
graphs have been proved to admit interval colorings; it is known,
for example, that trees, regular and complete bipartite graphs
\cite{AsrKam,Hansen,Kampreprint},  
bipartite graphs with maximum degree at most three \cite{Hansen},
doubly convex bipartite graphs
\cite{AsrDenHag,KamDiss}, grids \cite{GiaroKubale1}, 
and 
outerplanar bipartite graphs \cite{GiaroKubale2} have interval colorings. 
Additionally, all $(2,b)$-biregular graphs \cite{Hansen,HansonLotenToft,KamMir} and
 $(3,6)$-biregular graphs \cite{CarlJToft} admit interval colorings, where an
\emph{$(a,b)$-biregular} graph is a bipartite graph
where the vertices in one part all have degree $a$ and the vertices
in the other part all have degree $b$.
	
{\em Improper} (or {\em defective}) colorings was first considered
independently by
Andrews and Jacobson \cite{AndrewsJacobson}, Harary and Jones
\cite{HararyJones},
and Cowen et al. \cite{CowenGoddardJesurum}.
This coloring model is a well-known
generalization of ordinary graph coloring
with applications in various scheduling and assignment problems, see e.g.
the recent survey \cite{Wood}, or \cite{CowenGoddardJesurum}.
	 
	Motivated by scheduling and assignment problems with
	compactness requirements, but where a certain degree of conflict
	is acceptable,
	we consider {\em improper interval edge colorings} in this paper.
	An improper edge coloring of a graph is called an {\em improper interval
	(edge) coloring} if the colors on the edges incident with every vertex of
	the graph form a set of consecutive integers.
	This edge coloring model seems to have been first considered by
	Hudak et al. \cite{Hudak}, although their investigation
	has a different focus than ours.
	
	Note that unlike the case for interval colorings, every graph 
	trivially
	has an
	improper interval edge coloring.
	An improper interval coloring is {\em $k$-improper} if at most $k$ edges 
	with a common endpoint have the same color.
	We denote by $\muint(G)$ the smallest $k$ such that $G$
	has a $k$-improper interval edge coloring.
	The parameter $\muint(G)$ is called the 
	{\em interval coloring impropriety} (or just {\em impropriety}) of $G$.
	
	Improper interval edge colorings
	have immediate applications in scheduling problems,
	where an optimal schedule
	without waiting periods or idle times is desirable,
	but a certain level of conflict is allowed.
	For a bipartite graph $G$, representing a scheduling problem,
	the parameter $\muint(G)$ has a natural interpretation as the minimum
	degree of conflict necessary in a schedule with no waiting periods.
	Moreover, in view of the fact that not 
	every graph has an interval coloring,
	the parameter $\muint(G)$ may be viewed as a natural measure of how far from 
	being interval colorable a graph is. 
	
	Trivially, if $G$ has an interval coloring, then $\muint(G)=1$.
	In this paper,
	we provide several families of graphs with large impropriety;
	in particular, we prove that for each positive integer $k$,
	there is a graph $G$ with $\muint(G)=k$.
	
	We prove general upper bounds on $\muint(G)$
	and determine $\muint(G)$ exactly for some families of graphs $G$;
	in particular we prove that

	\begin{itemize}
	
		\item $\muint(G) \leq 2$ if $\Delta(G) \leq 5$,
		and $$\muint(G) \leq 
		\min\left\{2\left\lceil\frac{\Delta(G)}{\delta(G)}\right\rceil,
		\left\lceil\frac{\Delta(G)}{2}\right\rceil\right\},$$
		for any graph $G$ with $\Delta(G) \geq 6$,
	where $\Delta(G)$ and $\delta(G)$ denotes the maximum and minimum degree
	of a graph $G$, respectively;
	
	\item $\muint(G) \leq \left\lceil\frac{\Delta(G)}{4}\right\rceil$ 
	if $G$ is bipartite and has no vertices of degree three,
	and
	$$\muint(G) \leq \min\left\{\left\lceil\frac{\Delta(G)}{\delta(G)}\right\rceil,
		\left\lceil\frac{\Delta(G)}{3}\right\rceil\right\},$$
	for any bipartite graph $G$;
	
	\item $\muint(G) \leq \left\lceil\frac{r}{2}\right\rceil$ 
	if $G$ is a complete $r$-partite graph.
	
	\end{itemize}
	Furthermore, we conjecture that outerplanar graphs have impropriety at most
	$2$ and we prove this conjecture for graphs with maximum degree at most $8$.
	Finally, we consider the number of colors in an improper
	interval edge coloring and obtain a new upper bound on 
	the number of colors used in such a coloring.

	\section{Preliminaries}
	
	The degree of a vertex $v$ of a graph $G$ is denoted by $d_{G}(v)$.
	 $\Delta(G)$ and $\delta(G)$ denote the maximum and minimum
degrees of $G$, respectively. 
For two positive integers $a$ and $b$ with $a\leq b$, we denote by
$\left[a,b\right]$ the interval of integers $\left\{a,\ldots,b\right\}$. 
	
We shall need a classic result from factor theory.
A \emph{$2$-factor} of a multigraph $G$ (where loops are allowed) is a
$2$-regular spanning subgraph of $G$.
\begin{theorem} (Petersen's Theorem).
\label{th:Petersen}
 Let $G$ be a $2r$-regular
multigraph (where loops are allowed). Then $G$ has a decomposition
into edge-disjoint $2$-factors.
\end{theorem} 
	
If $\alpha $ is an edge coloring of $G$ and $v\in V(G)$, then
$S_{G}\left(v,\alpha \right)$ (or $S\left(v,\alpha \right)$) denotes
the set of colors appearing on edges incident to $v$; 
the
smallest and largest colors of the spectrum $S\left(v,\alpha \right)$
are denoted by $\underline S\left(v,\alpha \right)$ and $\overline S\left(v,\alpha \right)$, respectively.\\

The {\em chromatic index $\chi'(G)$} of a graph $G$
	is the minimum number  $t$ for which there exists a proper
	$t$-edge coloring of $G$.

	\begin{theorem} (Vizing's Theorem)
\label{th:Vizing}
For any 
graph $G$, $\chi'(G) = \Delta(G)$
	or $\chi'(G) = \Delta(G) +1$. 
	\end{theorem}
	A graph $G$  is said to be \emph{Class $1$} if $\chi'(G) = \Delta(G)$, and  
	\emph{Class $2$}
	 if $\chi'(G) = \Delta(G) +1$.
The next result gives a sufficient condition for a graph to be Class 1 (see, for example,  \cite{Fournier}).

\begin{theorem}
\label{th:Fournier}
	If $G$ is a graph where no two vertices of maximum degree are adjacent,
	then $G$ is Class 1.
\end{theorem}

Every bipartite graph is Class 1, as the following well-known
proposition, known as K\"onig's edge coloring theorem, states.

\begin{theorem} (K\"onig's edge coloring theorem)
\label{th:Konig}
	If $G$ is bipartite, then $\chi'(G) = \Delta(G)$.
\end{theorem}

We shall also need some preliminary results on interval edge coloring.
The following was proved by Hansen \cite{Hansen}.

\begin{theorem}
\label{th:Hansen}
	If $G$ is a bipartite graph with maximum degree $\Delta(G)\leq 3$,
	then $G$ has an interval coloring.
\end{theorem}

	\section{Improper interval edge colorings of some 
	non-interval-colorable graphs}

	In this section we determine the impropriety of some well-known families
	of graphs that in general do not admit interval colorings; in particular
	we describe constructions of bipartite graphs
	with arbitrarily large impropriety.
	
		\subsection{The impropriety of some non-interval-colorable graphs}
		
		Regular Class 1 graphs are trivially interval colorable,
		while no Class 2 graphs are \cite{AsrKam,Axen,GiaroKubaleMalaf2};
		however, all regular graphs have small impropriety.

		\begin{proposition}
		\label{prop:WheelReg}
			If $G$ is a regular graph, then 
			$$\muint(G) =
			\begin{cases}
				1, \text{ if $G$ is Class 1}, \\
				2, \text{ if $G$ is Class 2.}
			\end{cases}
			$$
		\end{proposition}
		\begin{proof}
	  Let $G$ be a regular graph. It is well-known
		that $G$ is interval colorable if and only if $G$ is Class 1.
		Hence, it suffices to prove that $\muint(G) \leq 2$;
		we shall give an explicit $2$-improper interval coloring
		of $G$.
		
		Suppose first that the vertex degrees of $G$ are even, say $d_G(v) = 2k$
		for every vertex $v \in V(G)$. By Petersen's theorem
		$G$ has a decomposition into $2$-factors $F_1,\dots,F_k$. By coloring
		all edges of $F_i$ by color $i$, $i=1,\dots,k$, we obtain a $2$-improper
		interval coloring of $G$.
		
		Suppose now that $d_G(v) =2k-1$ for all $v \in V(G)$.
		By taking two copies $G_1$ and $G_2$
		of $G$ and adding an edge between corresponding vertices of 
		$G_1$ and $G_2$,
		we obtain a $2k$-regular supergraph $H$. By the preceding paragraph,
		$H$ has a $2$-improper interval coloring. By taking the restriction of
		this coloring to $G_1$, it follows that $\muint(G) \leq 2$.		
		\end{proof}
		
		Note that Proposition \ref{prop:WheelReg} implies that
		for 
		cycles $C_n$ ($n\geq 3$) and complete graphs $K_n$ it holds that
		$$\muint(C_n) = \muint(K_n) =
			\begin{cases}
					1, \text{ if $n$ is even}, \\
					2, \text{ if $n$ is odd.}
			\end{cases}
		$$
		
	\begin{figure}[h]
	\begin{center}
		\includegraphics[width=20pc]{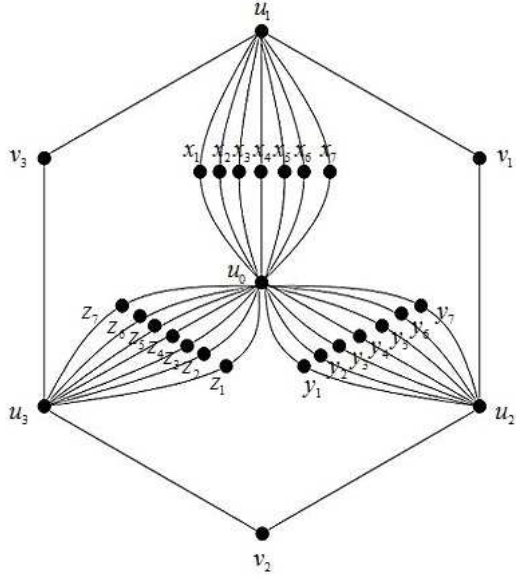}\\
		\caption{The graph $S_{7,7,7}$.}\label{fig1}
	\end{center}
	\end{figure}

Next, we consider generalizations of
two families of bipartite graphs with
no interval colorings introduced by Giaro et al. \cite{GiaroKubaleMalaf1}. 
For any $a,b,c\in \mathbb{N}$, define the graph $S_{a,b,c}$ as follows:
\begin{center}
$V(S_{a,b,c})=\{u_{0},u_{1},u_{2},u_{3},v_{1},v_{2},v_{3}\}\cup \{x_{1},\ldots, x_{a},y_{1},\ldots, y_{b},z_{1},\ldots,z_{c}\}$ and

\medskip

$E(S_{a,b,c})=\{u_{1}v_{1},v_{1}u_{2},u_{2}v_{2},v_{2}u_{3},u_{3}v_{3},v_{3}u_{1}\}\cup \{u_{0}x_{i},u_{1}x_{i}:1\leq i\leq a\}$\\
$\cup\{u_{0}y_{j},u_{2}y_{j}:1\leq j\leq b\}\cup\{u_{0}z_{k},u_{3}z_{k}:1\leq k\leq c\}$.
\end{center}
Figure \ref{fig1} shows the graph $S_{7,7,7}$.

Next we define a family of graphs $M_{a,b,c}$  ($a,b,c\in \mathbb{N}$).
We set
\begin{center}
$V(M_{a,b,c})=\{u_{0},u_{1},u_{2},u_{3}\}\cup \{x_{1},\ldots, x_{a},y_{1},\ldots, y_{b},z_{1},\ldots,z_{c}\}$ and 

$E(M_{a,b,c})=\{u_{0}x_{i},u_{1}x_{i},u_{2}x_{i}:1\leq i\leq a\}\cup\{u_{0}y_{j},u_{2}y_{j},u_{3}y_{j}:1\leq j\leq b\}$\\
$\cup\{u_{0}z_{k},u_{3}z_{k},u_{1}z_{k}:1\leq k\leq c\}$.
\end{center}
Figure \ref{fig2} shows the graph $M_{5,5,5}$.

Clearly, $S_{a,b,c}$ and $M_{a,b,c}$ are connected bipartite graphs. 
Giaro et al.  \cite{GiaroKubaleMalaf1} showed that 
the graphs $S_k=S_{k,k,k}$ and $M_l=M_{l,l,l}$ 
do not admit interval colorings if $k\geq 7$, and $l \geq 5$, respectively.
 
Here we shall prove that all graphs in the families 
$\{S_{a,b,c}\}$ and $\{M_{a,b,c}\}$ satisfy that $\muint(S_{a,b,c}) \leq 2$ and $\muint(M_{a,b,c}) \leq 2$, respectively.

\begin{figure}[h]
\begin{center}
\includegraphics[width=25pc]{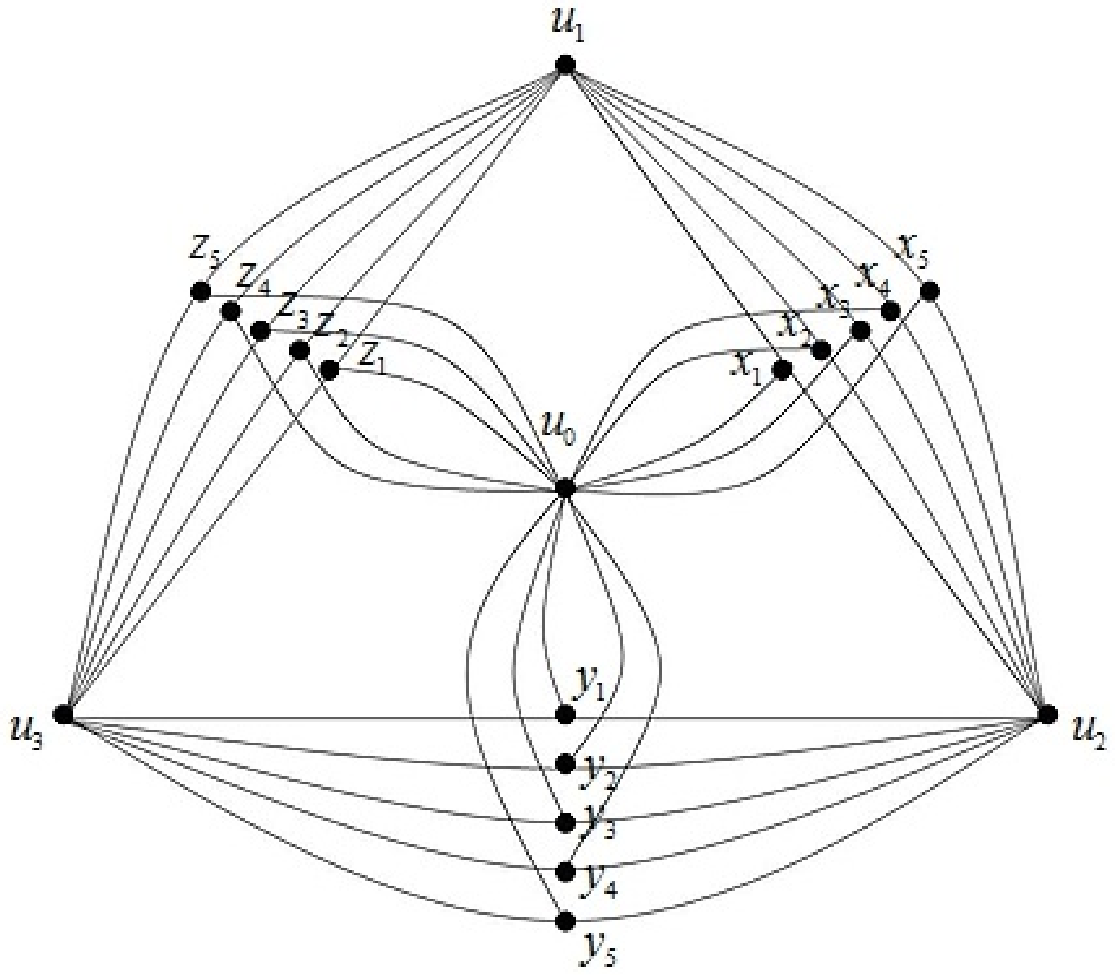}\\
\caption{The graph $M_{5,5,5}$.}\label{fig2}
\end{center}
\end{figure}

\begin{theorem}\label{rosettes}
For any $a,b,c\in \mathbb{N}$, $\muint(S_{a,b,c}) \leq 2$ and 
$\muint(M_{a,b,c}) \leq 2$.
\end{theorem}
\begin{proof}
Without loss of generality, we may assume that $a\leq b\leq c$.

We first construct an edge coloring $\alpha$ of the graph $S_{a,b,c}$. We define this coloring as follows:

\begin{description}
\item[($1$)] for $1\leq i\leq a$, let
$\alpha \left(u_{0}x_{i}\right)=\alpha \left(u_{1}x_{i}\right)=i$;

\item[($2$)] for $1\leq j\leq b$, let
$\alpha \left(u_{0}y_{j}\right)=\alpha \left(u_{2}y_{j}\right)=j$;

\item[($3$)] for $1\leq k\leq c$, let
$\alpha \left(u_{0}z_{k}\right)=\alpha \left(u_{3}z_{k}\right)=a+k$;

\item[($4$)] $\alpha \left(u_{1}v_{1}\right)=\alpha \left(v_{1}u_{2}\right)=1,\alpha \left(u_{2}v_{2}\right)=b+1,\alpha \left(v_{2}u_{3}\right)=b+2,$\
$\alpha \left(u_{3}v_{3}\right)=a,\alpha \left(v_{3}u_{1}\right)=a+1$.
\end{description}

It is straightforward that $\alpha$ is a $2$-improper interval coloring of $S_{a,b,c}$.

Next we define an edge coloring $\beta$ of the graph $M_{a,b,c}$ as follows: 

\begin{description}
\item[($1^{\prime}$)] for $1\leq i\leq a$, let
$\beta \left(u_{0}x_{i}\right)=i$;

\item[($2^{\prime}$)] for $1\leq i\leq a$, let
$\beta \left(u_{1}x_{i}\right)=\beta \left(u_{2}x_{i}\right)=i+1$;

\item[($3^{\prime}$)] for $1\leq j\leq b$, let
$\beta \left(u_{0}y_{j}\right)=j$;

\item[($4^{\prime}$)] for $1\leq j\leq b$, let
$\beta \left(u_{2}y_{j}\right)=\beta \left(u_{3}y_{j}\right)=j+1$;

\item[($5^{\prime}$)] for $1\leq k\leq c$, let
$\beta \left(u_{0}z_{k}\right)=a+k$;

\item[($6^{\prime}$)] for $1\leq k\leq c$, let
$\beta \left(u_{3}z_{k}\right)=\beta \left(u_{1}z_{k}\right)=a+k+1$.
\end{description}
It is easy to verify that $\beta$ is a $2$-improper interval coloring of $M_{a,b,c}$. 
%
%
We conclude that $\muint(S_{a,b,c}) \leq 2$ and $\muint(M_{a,b,c}) \leq 2$. 
~$\square$
\end{proof}

\bigskip

Lastly, let us consider two elementary classes of graphs
that have been proved not to always admit interval colorings.
Recall that a wheel graph $W_{n}$ on $n$ vertices  ($n\geq 4$) 
is defined as 
the	join of $C_{n-1}$ and $K_{1}$.
It is well-known that only few wheels are interval colorable, but they
all have small impropriety (which in fact is implicit in \cite{Hudak}).

\begin{proposition}	
If $W_{n}$ is a wheel graph on $n$ vertices, then
			$$\muint(W_{n}) =
			\begin{cases}
				1, \text{ if $n=4,7,10$}, \\
				2, \text{ otherwise.}
			\end{cases}
			$$
\end{proposition}	
\begin{proof}
			Let $W_{n}$ be a wheel graph. In 
			\cite{Axen,GiaroKubaleMalaf2}, it was shown that $W_{n}$ 
			has an interval coloring if and only if $n=4,7$ or $10$. 
			Hence, it suffices to prove that $\muint(W_n) \leq 2$; 
			this follows from a result in \cite{Hudak}: in fact
			the improper interval $(n-1)$-coloring of 
			$W_{n}$ described in the proof of Theorem 2.8 in \cite{Hudak} 
			is a $2$-improper interval coloring of $W_{n}$.
\end{proof}

In \cite{FengHuang}, the authors considered the
problem of constructing interval edge colorings
of so-called {\em generalized $\theta$-graphs};
a {\em generalized $\theta$-graph}, denoted by $\theta_{m}$, 
is a graph consisting of 
two vertices $u$ and $v$ together with
$m$ internally-disjoint $(u,v)$-paths, 
where $2\leq m<\infty$. 
These graphs also have small impropriety.

\begin{proposition}\label{theta}
For any $m \geq 2$, 
	$$\muint(\theta_{m}) =
			\begin{cases}
				1, \text{ if $\theta_{m}$ is not an Eulerian graph with an odd number of edges}, \\
				2, \text{ otherwise.}
			\end{cases}
	$$
\end{proposition}
\begin{proof}
In \cite{FengHuang}, it was proved that $\theta_{m}$ has an interval coloring if and only if it is not an Eulerian graph with an odd number of edges. Hence, it suffices to prove that $\muint(\theta_{m})\leq 2$; 
for $i=1,\dots,m$,
we color all edges of the $i$th path between $u$ and $v$ by color $i$.
Thus, trivially $\muint(\theta_{m})\leq 2$. 
\end{proof}


	\subsection{Graphs with large impropriety}
	
In this section we describe several families of graphs
with large impropriety. We begin our considerations with constructions based on subdivisions.\\
	
Let $G$ be a graph and $V(G)=\{v_{1},\ldots,v_{n}\}$. Define graphs
$S(G)$ and $\widehat{G}$ as follows:
\begin{center}
$V(S(G))=\{v_{1},\ldots,v_{n}\}\cup \{w_{ij}:v_{i}v_{j}\in E(G)\}$,
\end{center}
\begin{center}
$E(S(G))=\{v_{i}w_{ij},v_{j}w_{ij}:v_{i}v_{j}\in E(G)\}$,
\end{center}
\begin{center}
$V(\widehat{G})=V(S(G))\cup \{u\}$, $u\notin V(S(G))$,
$E(\widehat{G})=E(S(G))\cup \{uw_{ij}:v_{i}v_{j}\in E(G)\}$.
\end{center}

In other words, $S(G)$ is the graph obtained by subdividing every
edge of $G$, and $\widehat{G}$ is the graph obtained from $S(G)$ by
connecting every inserted vertex to a new vertex $u$. Note that
$S(G)$ and $\widehat{G}$ are bipartite graphs.

\begin{theorem}
\label{subdivisions} If $G$ is a connected graph and
\begin{center}
$\vert E(G)\vert
> k\left(1+ {\max\limits_{P\in \mathbf{P}}}{\sum\limits_{v\in
V(P)}}\left(d_{\widehat{G}}(v)-1\right)\right)$,
\end{center}
where $\mathbf{P}$ is a set of all shortest paths in $S(G)$
connecting vertices $w_{ij}$, then $\muint(\widehat{G})>k$.
\end{theorem}
\begin{proof}
Suppose, to the contrary, that $\widehat{G}$ has a $k$-improper interval
$t$-coloring $\alpha$; then $t\geq \frac{\vert E(G)\vert}{k}$,
because there is a vertex $u$ in $\widehat{G}$ that is adjacent to all vertices
in $V(S(G)) \setminus V(G)$.

Consider the vertex $u$, and let $w$ and $w^{\prime}$ be two vertices
adjacent to $u$ satisfying that $\alpha(uw)=\underline{S}(u,\alpha)=s$ and
$\alpha(uw^{\prime})=\overline{S}(u,\alpha)\geq s+\frac{\vert E(G)\vert}{k}-1$.
Since 
$\widehat{G}-u$ is connected, there is a
shortest path $P(w,w^{\prime})$ in $\widehat{G}-u$ joining $w$ with
$w^{\prime}$, where
\begin{center}
$P(w,w^{\prime})=x_{1},e_{1},x_{2},\ldots,x_{i},e_{i},x_{i+1},
\ldots,x_{r},e_{r},x_{r+1}$
and $x_{1}=w$, $x_{r+1}=w^{\prime}$.
\end{center}

Note that
\begin{center}
$\alpha(x_{i}x_{i+1})\leq s+\underset{j=1}{\overset{i}{\sum
}}(d_{\widehat{G}}(x_{j})-1)$ for $1\leq i\leq r$
\end{center}
and
$$\alpha(x_{r+1}u)=\alpha(w^{\prime}u)\leq
s+\underset{j=1}{\overset{r+1}{\sum }}(d_{\widehat{G}}(x_{j})-1).$$
Hence

$$s+\frac{\vert E(G)\vert}{k}-1\leq \overline{S}(u,\alpha)=\alpha(uw^{\prime})\leq
s+\underset{j=1}{\overset{r+1}{\sum }}(d_{\widehat{G}}(x_{j})-1)\leq
s+{\max\limits_{P\in \mathbf{P}}}{\sum\limits_{v\in V(P)}}\
\left(d_{\widehat{G}}(v)-1\right)$$

and thus
\begin{center}
$\vert E(G)\vert \leq k\left(1+{\max\limits_{P\in
\mathbf{P}}}{\sum\limits_{v\in V(P)}}\
\left(d_{\widehat{G}}(v)-1\right)\right)$,
\end{center}
which is a contradiction.
 
\end{proof}

\begin{corollary}
\label{cor_Kn} If $n^{2}-n>2k(2n+3)$, then $\muint(\widehat{K}_{n})>k$.
\end{corollary}

\begin{corollary}
\label{cor_Kmn} If $mn>k(m+n+5)$, then
$\muint(\widehat{K}_{m,n})>k$.
\end{corollary}

Our next construction uses techniques first described in \cite{PetrosKhacha} and generalizes the family of so-called Hertz graphs first described in \cite{GiaroKubaleMalaf1}.

Let $T$ be a tree and 
let $\mathcal{P}$ be the set of all paths in $T$. We set
$F(T)=\{v:v\in V(T)\wedge d_{T}(v)=1\}$, and 
define $M(T)$ as follows:

\begin{equation*}
M(T)=\max_{P \in \mathcal{P}} \{|E(P)| + |\{uw: uw \in E(T), u \in V(P),
w \notin V(P)\}|\}.
\end{equation*}

Now let us define the graph $\widetilde{T}$ as follows:
\begin{center}
$V(\widetilde{T})=V(T)\cup \{u\}$, $u\notin V(T)$,
$E(\widetilde{T})=E(T)\cup \{uv:v\in F(T)\}$.
\end{center}

Clearly, $\widetilde{T}$ is a connected graph with
$\Delta(\widetilde{T})=\vert F(T)\vert$. Moreover, if $T$ is a tree
in which the distance between any two pendant vertices is even, then
$\widetilde{T}$ is a connected bipartite graph.\\

\begin{theorem}
\label{trees} If $T$ is a tree and $\vert F(T)\vert >
k\left(M(T)+2\right)$, then $\muint(\widetilde{T})>k$.
\end{theorem}
\begin{proof}
Suppose, to the contrary, that $\widetilde{T}$ has a $k$-improper interval
$t$-coloring $\alpha$ for some $t\geq \frac{\vert F(T)\vert}{k}$.

Consider the vertex $u$. Let $v$ and $v^{\prime}$ be two vertices
adjacent to $u$ such that $\alpha(uv)=\underline{S}(u,\alpha)=s$ and
$\alpha(uv^{\prime})=\overline{S}(u,\alpha)\geq s+\frac{\vert F(T)\vert}{k}-1$. Since $\widetilde{T}-u$ is a tree, there is a unique path
$P(v,v^{\prime})$ in $\widetilde{T}-u$ joining $v$ with
$v^{\prime}$, where

\begin{center}
$P(v,v^{\prime})=x_{1},e_{1},x_{2},\ldots,x_{i},e_{i},x_{i+1},\ldots,
x_{r},e_{r},x_{r+1}$ and $x_{1}=v$, $x_{r+1}=v^{\prime}$.
\end{center}

Note that
\begin{center}
$\alpha(x_{i}x_{i+1})\leq s+1+\underset{j=1}{\overset{i}{\sum
}}(d_{T}(x_{j})-1)$ for $1\leq i\leq r$.
\end{center}

From this, we have

$$\alpha(x_{r}x_{r+1})=\alpha(x_{r}v^{\prime})\leq
s+1+\underset{j=1}{\overset{r}{\sum
}}(d_{T}(x_{j})-1)\leq s+M(T).$$

Hence

$$s+\frac{\vert F(T)\vert}{k}-1\leq \overline{S}(u,\alpha)=\alpha(uv^{\prime})\leq
s+1+M(T),$$ 
and thus $\vert F(T)\vert\leq k\left(M(T)+2\right)$, which is a contradiction.
 
\end{proof}

\begin{corollary}
\label{cor_bip} If $T$ is a tree in which the distance
between any two pendant vertices is even and $\vert F(T)\vert >
k\left(M(T)+2\right)$, then the bipartite graph $\widetilde{T}$ has no $k$-improper interval coloring.
\end{corollary}

The {\em deficiency} of a graph $G$ is the minimum number of edges
whose removal from $G$ yields a graph with an interval coloring. Thus,
the deficiency of a graph is another measure of how far from being
interval colorable a graph is.

As mentioned above,
our constructions by trees generalize the so-called
Hertz's graphs $H_{p,q}$, first described in
\cite{GiaroKubaleMalaf1}. Hertz's graphs
are known to have a high deficiency,
so let us specifically 
consider the impropriety of such graphs.

In \cite{GiaroKubaleMalaf1} the Hertz's graph 
$H_{p,q}$ $(p,q\geq 2)$ was defined as
follows:
\begin{center}
$V(H_{p,q})=\left\{a,b_{1},b_{2},\ldots,b_{p},d\right\}\cup
\left\{c_{j}^{(i)}:1\leq i\leq p, 1\leq j\leq q\right\}$ and
\end{center}
\begin{center}
$E(H_{p,q})=E_{1}\cup E_{2}\cup E_{3}$,
\end{center}
where
\begin{center}
$E_{1}=\left\{ab_{i}:1\leq i\leq p\right\}$,
$E_{2}=\left\{b_{i}c_{j}^{(i)}:1\leq i\leq p, 1\leq j\leq
q\right\}$, $E_{3}=\left\{c_{j}^{(i)}d:1\leq i\leq p, 1\leq j\leq q
\right\}$.
\end{center}

The graph $H_{p,q}$ is bipartite with maximum degree $\Delta (H_{p,q})=pq$ and
$\vert V(H_{p,q})\vert=pq+p+2$. 
We are now able to prove the following result; 
our main result of this section.

\begin{theorem}\label{for_any_k} For any $k \in \mathbb{N}$, there exists a bipartite graph $G$ such that $\muint(G) =k$.
\end{theorem}
\begin{proof} 
For a given $k$, choose $p$ so that $p\geq 2k^{2}-1$. Let us consider the tree $T=H_{p,k}-d$. Since $M(T)=p+2k$, $\vert F(T)\vert=pk$ and 
the graph $H_{p,k}$ is isomorphic to $\widetilde{T}$, by Theorem \ref{trees}, we obtain that $\muint(H_{p,k})>k-1$. On the other hand, let us define an edge coloring $\alpha$ of $H_{p,k}$ as follows:

\begin{description}
\item[($1$)] for $1\leq i\leq p$, let
$\alpha \left(ab_{i}\right)=i+1$;

\item[($2$)] for $1\leq i\leq p$ and $1\leq j\leq k$, let
$\alpha \left(b_{i}c_{j}^{(i)}\right)=\alpha \left(c_{j}^{(i)}d\right)=i$.
\end{description}

It is easy to verify that $\alpha$ is a $k$-improper interval coloring of $H_{p,k}$; thus $\muint(H_{p,k})\leq k$.
\end{proof}

\bigskip	

In the last part of this section we use finite projective planes for constructing
bipartite graphs with large impropriety. This family of graphs was first described in \cite{PetrosKhacha}.

Let $\pi(n)$ be a finite projective plane of order $n\geq 2$, $\{1,2,\ldots,n^{2}+n+1\}$ be the set of points 
and $L= \{l_1,l_2,\dots, l_{n^2+n+1}\}$ the set
of lines of $\pi(n)$.
Let $A_i=\{k \in l_i : 1 \leq k \leq n^2+n+1\}$ for every $1 \leq i \leq n^2+n+1$;
then $|A_i|=n+1$ for every $i$, and $A_i \neq A_j$ if $i \neq j$.
For a sequence of $n^2+n+1$ integers $r_{1}, r_2,\ldots, r_{n^{2}+n+1}\in \mathbb{N}$ ($r_{1}\geq \ldots\geq r_{n^{2}+n+1}\geq 1$), we define the graph
$Erd(r_{1},\ldots,r_{n^{2}+n+1})$ as follows:

\begin{center}
$V(Erd(r_{1},\ldots,r_{n^{2}+n+1}))=\{u\}\cup\{1,\ldots,n^{2}+n+1\}$\\
$\cup \left\{v^{(l_{i})}_{1},\ldots,v^{(l_{i})}_{r_{i}}:
1\leq i\leq n^{2}+n+1\right\}$,\\
\end{center}
\begin{equation*}
E(Erd(r_{1},\ldots,r_{n^{2}+n+1}))=
\bigcup_{i=1}^{n^{2}+n+1} \left(
\left\{uv^{(l_{i})}_{1},\ldots,uv^{(l_{i})}_{r_{i}} 
\right\}\cup 
\left\{v^{(l_{i})}_{1}k,\ldots,v^{(l_{i})}_{r_{i}}k : 
k\in A_{i}, 
\right\} \right).
\end{equation*}
The graph $Erd(r_{1},r_2,\ldots,r_{n^{2}+n+1})$ is a 
connected bipartite
graph with  
$n^2+n+2+\underset{i=1}{\overset{n^{2}+n+1}{\sum
}}r_{i}$ 
vertices
and maximum degree $\underset{i=1}{\overset{n^{2}+n+1}{\sum}}r_{i}$.

Note that the above  graph with parameters $n=3$ and $r_1=r_2=\dots=r_{13}=1$ was described in 1991 by Erd\H{o}s \cite{JensenToft}.

\begin{theorem}
\label{Erdos} If $\underset{i=n+2}{\overset{n^{2}+n+1}{\sum
}}r_{i}-(k-1)\underset{i=1}{\overset{n+1}{\sum }}r_{i}>2k(n+1)$, then $\muint\left(Erd(r_{1},\ldots,r_{n^{2}+n+1})\right)>k$.
\end{theorem}
\begin{proof}
Suppose, to the contrary, that the graph
$G=Erd(r_{1},\ldots,r_{n^{2}+n+1})$ has a $k$-improper interval
$t$-coloring $\alpha$ for some $t\geq \frac{\underset{i=1}{\overset{n^{2}+n+1}{\sum
}}r_{i}}{k}$.

Consider the vertex $u$ of $G$, and let $v_{p}^{(l_{i_{0}})}$ and
$v_{q}^{(l_{j_{0}})}$ be two vertices adjacent to $u$ such that
$\alpha\left(uv_{p}^{(l_{i_{0}})}\right)=\underline{S}(u,\alpha)=s$ and
$\alpha\left(uv_{q}^{(l_{j_{0}})}\right)=\overline{S}(u,\alpha)\geq s+\frac{\underset{i=1}{\overset{n^{2}+n+1}{\sum
}}r_{i}}{k}-1$.

If $l_{i_{0}}=l_{j_{0}}$, then, by the construction of $G$
there exists $k_{0}$ such that
$k_{0}v_{p}^{(l_{i_{0}})}$, $k_{0}v_{q}^{(l_{j_{0}})}\in
E(G)$.
If, on the other hand $l_{i_{0}}\neq l_{j_{0}}$,
then $l_{i_{0}}\cap l_{j_{0}}\neq \emptyset$;
so again,
by the construction of $G$, 
there exists $k_{0}$ such that
$k_{0}v_{p}^{(l_{i_{0}})}$, $k_{0}v_{q}^{(l_{j_{0}})}\in
E(G)$.

Now, 
we have 
$d\left(v_{p}^{(l_{i_{0}})}\right)=d\left(v_{q}^{(l_{j_{0}})}\right)=n+2$
and
\begin{center}
$\alpha\left(k_{0}v_{p}^{(l_{i_{0}})}\right)\leq
s+d\left(v_{p}^{(l_{i_{0}})}\right)-1=s+n+1$,
\end{center}
and thus
\begin{center}
$\alpha\left(k_{0}v_{q}^{(l_{j_{0}})}\right)\leq
s+n+1+d(k_{0})-1\leq s+n+\underset{i=1}{\overset{n+1}{\sum }}r_{i}$.
\end{center}

This implies that
\begin{center}
$s+\frac{\underset{i=1}{\overset{n^{2}+n+1}{\sum
}}r_{i}}{k}-1\leq \alpha\left(uv_{q}^{(l_{j_{0}})}\right)=\overline{S}(u,\alpha)\leq
s+n+\underset{i=1}{\overset{n+1}{\sum
}}r_{i}+d\left(v_{q}^{(l_{j_{0}})}\right)-1=s+2n+1+\underset{i=1}{\overset{n+1}{\sum
}}r_{i}$.
\end{center}

Hence, $\underset{i=n+2}{\overset{n^{2}+n+1}{\sum }}r_{i}-(k-1)\underset{i=1}{\overset{n+1}{\sum }}r_{i}\leq
2k(n+1)$, which is a contradiction.
\end{proof}

Using Theorem \ref{Erdos} we can generate infinite families of graphs
with large impropriety. For example, if 
$r_1 = r_2 = \dots = r_{n^2+n+1} = r$, where $r$ is some constant, then $\muint\left(Erd(r_{1},\ldots,r_{n^{2}+n+1})\right)>k$ if $n^{2}r-(k-1)r(n+1)>2k(n+1)$.
	
	
\section{Upper bounds on the impropriety of graphs}
	
	In this section, we give general upper bounds on $\muint(G)$
	for several different families of graphs.
	There is a prominent line of research on interval colorings
	of bipartite graphs;
	we begin this section by considering improper interval colorings
	of bipartite graphs.

	\subsection{Bipartite graphs}
	
	 As mentioned above, Hansen \cite{Hansen} proved 
	that if $G$ is bipartite and 
	satisfies that $\Delta(G) \leq 3$, then $G$ has an interval coloring,
	while the question 
	of interval colorability
	for bipartite graphs of maximum degree
	$4$ is open. However, using Hansen's result and K\"onig's edge coloring
	theorem, we deduce the following upper bound.

	\begin{theorem}
		If $G$ is bipartite, then 
		\begin{itemize}
			
			\item[(i)] $\muint(G) \leq \left\lceil 
			\frac{\Delta(G)}{\delta(G)} \right\rceil$;
			
			\item[(ii)] $\muint(G) \leq \left\lceil  
			\frac{\Delta(G)}{3} \right\rceil$.
			\end{itemize}
		\end{theorem}
		
		\begin{proof}
			Let $G$ be a bipartite graph.
			To prove (i), we construct a new bipartite graph $H$ from $G$
			by proceeding in the following way: for every vertex $v$ of degree
			at least $\delta(G)+1$, we split $v$ into as many vertices
			of degree $\delta(G)$ as possible, and one vertex of degree less than
			$\delta(G)$. Since the graph $H$ has maximum degree $\delta(G)$, 
			by K\"onig's edge coloring theorem, it
			has a proper $\delta(G)$-edge coloring $\varphi$. Let $\varphi_G$
			be the coloring of $G$ induced by this coloring of $H$.
			Since each vertex of $G$ is split into at most
			$\left\lceil 
			\frac{\Delta(G)}{\delta(G)} \right\rceil$ vertices,
			the coloring $\varphi_G$ is a $\left\lceil 
			\frac{\Delta(G)}{\delta(G)} \right\rceil$-improper interval
			coloring of $G$ using $\delta(G)$ colors.
			
			Part (ii) can be proved similarly to part (i), except that 
			we apply Theorem \ref{th:Hansen} to the graph
			obtained from $G$ by splitting
			every vertex of
			$G$ into vertices of degree at most three.
		\end{proof}
		
		If $G$ is bipartite, and, in addition,
		has no vertices of degree $3$, then
		we have the following:
		
		\begin{proposition}
			If $G$ is bipartite and has no vertices of degree three,
			then $\muint(G) \leq \left\lceil 
			\frac{\Delta(G)}{4} \right\rceil$.
		\end{proposition}
		\begin{proof}
			We proceed as in the preceding proof. From the
			bipartite graph $G$, we construct a graph $G'$ by splitting
			every vertex of degree at least five into as many vertices
			of degree four as possible, and one vertex of degree at most three.
			From $G'$, we construct a graph $G''$ with even vertex degrees
			by taking two
			copies of the graph $G'$ and joining any two corresponding vertices
			of degree three or one by an edge.
			Finally, we construct a
			$4$-regular multigraph $H$ by adding a loop at every
			vertex of degree two. Now, by Petersen's theorem,
			$H$ has a decomposition into two $2$-factors $F_1$ and $F_2$.
			In $G''$, the subgraph $F_i$ corresponds to a collection
			of even cycles, $i=1,2$. By coloring the edges of every
			cycle in $G''$ corresponding to a cycle of $F_1$
			alternately by colors $1,2$; and the
			edges of every cycle corresponding to a cycle of $F_2$
			alternately by colors $3,4$,
			we obtain an interval edge coloring $\varphi$ of $G''$, where
			every vertex of degree $2$ has colors $1$ and $2$, or $3$ and $4$,
			on its incident edges.
			
			Since there are no vertices of degree three in $G$,
			and each vertex of $G$ is split into at most $\left\lceil 
			\frac{\Delta(G)}{4} \right\rceil$ different vertices in $G'$,
			the coloring $\varphi$ corresponds to a 
			$\left\lceil  \frac{\Delta(G)}{4} \right\rceil$-improper
			interval edge coloring of $G$.
		\end{proof}

		For bipartite graphs with small vertex degrees we deduce some
		consequences of the above results.
		
		\begin{corollary}
			If $G$ is bipartite and $\Delta(G) \leq 6$
			then $\muint(G) \leq 2$.
		\end{corollary}

		\begin{corollary}
			If $G$ is bipartite, Eulerian and $\Delta(G) \leq 8$
			then $\muint(G) \leq 2$.
		\end{corollary}

		In general, for $k\geq 2$, 
		it would be interesting
		to determine or bound the least integer $f_{\text{bip}}(k)$ for which there
		exists a graph $G$ with maximum degree $f_{\text{bip}}(k)$
		satisfying $\muint(G) =k$. 
		Even the case $k=2$ of this problem is open.
		It is known, however, that 
		$4 \leq f_{\text{bip}}(2) \leq 11$, see e.g. \cite{PetrosKhacha}.
		Moreover, by the results of Hertz graphs, $f_{\text{bip}}(3) \leq 51$,
		and by the above corollary
		$f_{\text{bip}}(3) \geq 7$.

\subsection{General graphs}
					
		Let us now deduce some upper bounds for general graphs.
		As for bipartite graphs, we define $f(k)$ as the smallest
		integer such that there exists a graph $G$ with maximum
		degree $f(k)$ and $\muint(G)=k$.
		The smallest graphs with impropriety $2$ are odd cycles;
		thus $f(2)=2$.
		
		We believe that the following question is of particular interest:

		\begin{problem}
		\label{maxdegree}
			Determine $f(3)$, that is, determine the least integer
			$\Delta$, such that there is a graph $G$ with maximum degree $\Delta$
			satisfying $\muint(G)=3$.
		\end{problem}

		The following result shows that $f(3) > 5$ in Problem
		\ref{maxdegree}.
		
		\begin{theorem}
		\label{prop:maxdeg5}
			If $G$ is a graph with $\Delta(G) \leq 5$, then
			$\muint(G) \leq 2$.			
		\end{theorem}

		\begin{proof}
		If $G$ has maximum degree $2$, then trivially $\muint(G) \leq 2$.			
		Let us now consider the case when $G$ satisfies $3 \leq \Delta(G) \leq 4$;
		again, we shall use Petersen's $2$-factor theorem.
			From $G$ we form a new graph $G'$ by taking two copies
			of $G$ and adding an edge between any two corresponding vertices
			of odd degree. From $G'$ we form a new $4$-regular graph $H$ by adding a
			loop at every vertex of degree $2$ in $G'$. By Petersen's theorem,
			$H$ has a decomposition into two $2$-factors $F_1$ and $F_2$.
			In $G'$, $F_i$ corresponds to a collection $A_i$ of cycles, $i=1,2$.
			By coloring edges of all cycles of $A_i$ by color $i$, we obtain
			a $2$-improper interval coloring $\varphi$
			of $G'$, and the result now follows
			by coloring $G$ according to the restriction of $\varphi$ to
			one of the copies of $G$ in $G'$.
		
		  Let us now consider the case when $\Delta(G) = 5$.
			Let $G_{5}$ be the subraph of $G$ induced by
			the vertices of degree $5$ in $G$. Let $M$ be a 
			maximum matching in $G_{5}$. 		Since $M$ is maximum,
			the graph $H= G - M$ either has maximum degree $4$ or 
			no two vertices of degree $5$ in $H$ are adjacent.
			It follows that
			$H$ has a proper $5$-edge coloring; 
			in the former case by Vizing's theorem,
			and in the latter case $H$ is Class 1 by Theorem \ref{th:Fournier}. 
			If $H$ has maximum degree $5$, then we set $H'=H-M'$, 
			where $M'$ is a matching covering all vertices with degree $5$ in $H$.
			If $H$ has maximum degree $4$, then we set $H'=H$ and $M' = \emptyset$.

			Now, by the argument in the preceding paragraph, 
			$H'$ has a $2$-improper interval coloring $\alpha$ with colors $1,2$ 
			such that for any vertex $v$ with degree $4$ or $3$, 
			$S(v,\alpha)=\{1,2\}$, and for any vertex $v$ with degree $2$ or $1$, 
			$S(v,\alpha)=\{1\}$ or $S(v,\alpha)=\{2\}$.
			
			Let us 
			define a new edge coloring $\beta$ of $G-M'$ 
			by coloring the edges of $M$
			as follows: 
			for every $e\in E(G-M')$, let 
		\begin{center}
        $\beta(e)=
        \left\{
        \begin{tabular}{ll}
        $\alpha(e)$, & if $e\in E(H')$,\\
        $3$, & if $e\in M$.\\
        \end{tabular}%
        \right.$
        \end{center}	
				Since $M$ is a maximum matching in $G_5$, the coloring $\beta$
				is a $2$-improper interval $3$-coloring of $G-M'$.
	   	From $\beta$ we define an edge coloring 
			$\gamma$ of $G$ as follows: for every $uv\in E(
			G)$, let 
		\begin{center}
        $\gamma(uv)=
        \left\{
        \begin{tabular}{ll}
        $\beta(uv)$, & if $uv\in E(G-M')$,\\
        $3$, & if $uv\in M'$ and $\overline{S}(u,\beta)=\overline{S}(v,\beta)=2$,\\
        $0$, & otherwise.\\
        \end{tabular}%
        \right.$
        \end{center}
				If there is an edge $e_{0}$ such that $\gamma(e_{0})=0$, 
				then we define an edge coloring $\gamma'$ of $G$ as follows: 
				$\gamma'(e)=\gamma(e)+1$ for every $e\in E(G)$. 
				It is straightforward that 
				if this holds, then
				$\gamma'$ is a $2$-improper 
				interval $4$-coloring of $G$; otherwise $\gamma$ is a $2$-improper
				interval $3$-coloring of $G$. Thus, $\muint(G) \leq 2$. 
    \end{proof}
		     We note that the upper bound in Theorem  \ref{prop:maxdeg5} is 
		in fact sharp,
		since any regular Class 2 graph is not interval colorable.
		    
		\bigskip
		
		It also seems that graphs $G$ whose vertex degrees are sufficiently
		concentrated satisfy $\muint(G) \leq 2$; for instance,
		as pointed out above,
		any regular graph $G$ satisfies that $\muint(G)\leq 2$.
		We strengthen this observation slightly as follows.
		
		\begin{proposition}
		\label{prop:maxmindeg}
			If $G$ is a graph with $\Delta(G)- \delta(G) \leq 1$, then
			$\muint(G) \leq 2$.			
		\end{proposition}
		\begin{proof}
			Let $G$ be a graph satisfying $\Delta(G)- \delta(G) \leq 1$,
			and let $G_{\Delta}$ be the subraph of $G$ induced by
			the vertices of maximum degree in $G$. 
			By the preceding proposition, we may assume that $\Delta(G) > 5$.
			Let $M$ be a maximum matching in $G_{\Delta}$. Since $M$ is maximum,
			the graph $H= G- M$ either has maximum degree $\Delta(G)-1$ or 
			no two vertices of degree $\Delta(G)$ in $H$ are adjacent.
			Then $H$ has a proper $\Delta(G)$-edge 
			coloring $\varphi$; in the former case
			by Vizing's theorem, and in the latter case $H$ is Class 1
			by Theorem \ref{th:Fournier}. 
			
			Let $M_1$ be the set of edges with color $1$ under $\varphi$, 
			$M_{\Delta}$ be the edges of color $\Delta(G)$ under $\varphi$,
			and consider the edge-induced subgraph $H[M_1 \cup M_\Delta]$.
			Since $\delta(H) \geq \Delta(G)-1$, this graph is
			a spanning subgraph, and, furthermore,
			every component of this graph is an even cycle or a path.
			Let $D$ be an orientation of $H[M_1 \cup M_\Delta]$
			where every vertex has indegree at most $1$ and outdegree
			at most $1$. 
			
			We define a new proper edge coloring of $H$ from $\varphi$
			by recoloring some of the edges in $H[M_1 \cup M_{\Delta}]$
			in the following way: for every arc $(a,b)$ of $D$, 
			if there is a color $c \in \{2,\dots, \Delta(G)-1\}$
			which does not appear on an edge incident with $b$ under $\varphi$,
			then we recolor the edge $ab$ with color $c$;
			if there is no such color $c$, then we retain the color of
			the edge $ab$. Denote the obtained coloring by $\varphi'$.
			Finally, we extend the coloring $\varphi'$
			to a coloring $\alpha$ of $G$ by coloring every edge of $M$ by color
			$\Delta(G)$.
			
			Let us prove that $\alpha$ is a $2$-improper interval edge
			coloring of $G$. The color $\Delta(G)$ appears at most twice
			at a vertex of $G$, and if two edges of
			$H[M_1 \cup M_{\Delta}]$, both of which are
			incident with a common vertex $u$,
			are recolored by the same color $j \in \{2,\Delta(G)-1\}$,
			then $j$ does not appear on any edge incident with $u$
			under $\varphi$. Hence, every color appears at most twice
			at any vertex of $G$.
			
			Suppose now that the colors on the edges incident with some
			vertex $v$ of $G$ under $\alpha$ does not form an interval.
			Since 
			$\alpha$ uses $\Delta(G)$ colors, this means that there is
			some color $j \in \{2,\dots, \Delta(G)-1\}$ that does not appear
			on an edge incident with $v$ under $\alpha$. Moreover, 
			since $\alpha$ is obtained from $\varphi$
			by recoloring only edges of color $1$ or $\Delta(G)$,
			$j$ does not appear at $v$ under $\varphi$.
			Now, since $\varphi$ is a proper $\Delta(G)$-edge coloring
			of $H$ and $\delta(H) \geq \Delta(G)-1$, we must have
			$d_H(v) = \Delta(G)-1$; and so $v$ is incident with an edge
			colored $i$ under $\varphi$, for every 
			$i \in \{1,\dots, j-1, j+1,\dots,\Delta(G)\}$;
			in particular, $v$ has degree $2$ in $H[M_1 \cup M_{\Delta}]$,
			and therefore one of the edges in $H[M_1 \cup M_{\Delta}]$
			would have been recolored $j$ in the process of constructing
			$\varphi'$  of $\varphi$. This is a contradiction, and so
			it follows that $\alpha$ is a $2$-improper interval edge coloring
			of $G$.			
		\end{proof}

		Using the preceding proposition, we can prove the following,
		by splitting vertices.
		
		\begin{proposition}
			If $G$ is a graph, then
			$\muint(G) \leq 2\left\lceil
			\frac{\Delta(G)}{\delta(G)} \right\rceil$.
		\end{proposition}
		\begin{proof}
			We proceed as before: from $G$ we form a new graph $G'$
			by splitting every vertex of degree at least $\delta(G)+1$
			into as many vertices of degree exactly $\delta(G)$ as possible, 
			and one vertex of degree at most $\delta(G)$.
			Let $H$ be a $\delta(G)$-regular supergraph of $G'$.
			By Proposition \ref{prop:maxmindeg}, $H$ has a $2$-improper
			interval edge coloring. 
			This coloring induces a $2\left\lceil
			\frac{\Delta(G)}{\delta(G)} \right\rceil$-improper interval coloring
			of $G$.
		\end{proof}
		
		Finally, we have the following general upper bound.

		\begin{theorem}
		\label{prop:genbound}
			If $G$ is a graph with $\Delta(G) \geq 6$,
			then $\muint(G) \leq \left\lceil
			\frac{\Delta(G)}{2} \right\rceil$.
		\end{theorem}
		\begin{proof}
			Since any regular graph has impropriety at most $2$,
			it suffices to consider the case when $\delta(G) < \Delta(G)$.
			Furthermore, without loss of generality, we assume that
			$G$ is connected.
			
			Let $H$ be the graph obtained by taking two copies of $G$ and
			adding an edge between any two corresponding vertices of odd degree.
			We shall consider $G$ as a subgraph of $H$.
			
			Since all vertex degrees in $H$ are even, it has an Eulerian circuit
			$T$. By coloring all edges of $T$ by $1$ and $2$ alternately
			along $T$,
			we obtain an improper interval coloring $\varphi$ of $H$. 
			Let $\alpha$ be the improper interval coloring of $G$
			induced by $\varphi$. If every vertex of $G$ is incident with
			at most $\lceil \frac{\Delta(G)}{2} \rceil$ edges with the same
			color, then the desired result follows, so assume that this does not
			hold. Then there is a vertex $v$ which is incident with exactly
			$\lceil \frac{\Delta(G)}{2} \rceil +1$ edges with the same color
			under $\alpha$,
			say $1$. 
			Indeed, since
			the edges of $T$ are colored
			alternately by colors $1$ and $2$, 
			$v$ must be the first vertex of the Eulerian circuit $T$ in $H$.
			 Without loss of generality, we assume that $v$ is a vertex
			of maximum degree in $G$.

			Let us first consider the case when $\Delta(G)$ is odd, that is,
			$\Delta(G) = \Delta(H)-1$.
			Let $T_1$ be a shortest subtrail of $T$ from
			$v$ to a vertex $u$ of degree at most $\Delta(G)-1$ in $G$.
			We define a new coloring $\varphi'$ of $H$ from $\varphi$
			by recoloring the edges on $T_1$ in the following way:
			we set 
					$$\varphi'(e) = 
			\begin{cases}
					1, & \text{ if $\varphi(e)=2$ and $e \in E(T_1)$}, \\
					2, & \text{ if $\varphi(e)=1$ and $e \in E(T_1)$}, \\
					\varphi(e), & \text{ if $e \notin E(T_1)$}.
			\end{cases}
			$$
			Since 
			 $d_G(u) \leq \Delta(G)-1$, 
			and all vertices of $H$ except $v$ and $u$ are incident
			with equally many edges of color $i$ under $\varphi'$ as under
			$\varphi$, $i=1,2$,
			it follows that the restriction of $\varphi'$
			to $G$ is a  $\lceil
			\frac{\Delta(G)}{2}\rceil$-improper interval coloring of $G$.
			
			Let us now consider the case when $\Delta(G)$ is even,
			i.e. $\Delta(H)=\Delta(G)$.
			Let $T_1$ be a shortest subtrail of $T$ from $v$ 
			to a vertex $x$ of degree at most $\Delta(G)-1$ in $G$, and 
			suppose $e_1$ is the last
			edge of $T_1$. 
			We consider some different cases.
			\begin{itemize}
			
				\item[(a)] If $e_1 \notin E(G)$ or $d_G(x) \leq \Delta(G)-2$,
				then we define a new coloring $\varphi'$
				from $\varphi$ by recoloring
				all edges of $T_1$ by setting $\varphi'(e) =1$ if $\varphi(e)=2$,
				$\varphi'(e)=2$ if $\varphi(e)=1$, and retaining the color
				of every other edge of $H$. The coloring $\alpha'$ of $G$
				induced by $\varphi'$ is a 
				$\lceil\frac{\Delta(G)}{2}\rceil$-improper interval coloring of $G$.
				
				\item[(b)] If $e_1 \in E(G)$, $d_G(x) = \Delta(G)-1$,
				$\varphi(e_1) = 1$ $(2)$, and $x$
				is incident with an edge in $E(H)\setminus E(G)$ 
				of color $2$ $(1)$ under $\varphi$, then we proceed as in (a).
				
				\item[(c)] If $e_1 \in E(G)$, $d_G(x) = \Delta(G)-1$,
				$\varphi(e_1) =1$ $(2)$, and $x$
				is incident with an edge $e_2 \neq e_1$ in $H$ of color $1$ $(2)$
				under $\varphi$
				that is not in $G$, then we proceed as follows:
				let $T_2$ be the subtrail of $T$ beginning with $v$ whose
				last edge is $e_2$. By proceeding as in (a) and switching colors
				on $T_2$, and taking the restriction of the obtained coloring
				to $G$, we obtain a 
				$\lceil\frac{\Delta(G)}{2}\rceil$-improper interval coloring of $G$.
				\end{itemize}
			
			\end{proof}
		
		\subsection{Outerplanar graphs}
		
		\bigskip
		
		In this section we consider outerplanar graphs.
		We do not know of any outerplanar graph $G$
		with $\muint(G) \geq 3$; in fact, 
		we believe that there is no such graph.
		
		\begin{conjecture}
		\label{conj:outerplanar}
			For any outerplanar graph $G$, $\muint(G) \leq 2$.
		\end{conjecture}
		
		Since there are examples of outerplanar graphs
		with no interval edge coloring,
		the upper bound in Conjecture \ref{conj:outerplanar} would be sharp
		if true.
		Next, we shall prove that this conjecture holds for
		graphs with maximum degree at most eight.
		
		\begin{proposition}
		\label{prop:outerplanar}
			If $G$ is an outerplanar graph and $\Delta(G) \leq 8$,
			then $\muint(G) \leq 2$.		
		\end{proposition}
		
		For the proof of this result, we shall use the well-known
		fact that an outerplanar graph is Class 1 unless it is an
		odd cycle \cite{Fiorini}.
		
		\begin{proof}
			Since a graph $G$ has a $k$-improper interval coloring
			if every block of $G$ has a $k$-improper interval coloring,
			it suffices to consider the case when $G$ is $2$-connected.
		
			Consequently, assume that $G$ is a $2$-connected; then it
			has a Hamiltonian cycle $C$. The graph $G-E(C)$
			has maximum degree at most $6$.
			If $|V(C)|$ is even, then we define a  proper edge coloring $\varphi$
			of $C$ by coloring its edges alternately by colors $2$ and $3$.
			If $|V(C)|$ is odd, then we define $\varphi$ in the following way:
			it is well-known that every $2$-connected
			outerplanar has a vertex $v$ of degree $2$; we color the edges of $C$
			alternately by colors $2$ and $3$, and beginning and ending with
			color $2$ at the edges incident with $v$.
			
			Now, consider the graph $H=G-E(C)$. Since $H$ is 
			an outerplanar graph (or consisting of several outerplanar
			components),
			it has a proper edge coloring $\alpha$ using at most $6$ colors
			$1,\dots,6$.
			From $\alpha$, we define a new edge coloring $\alpha'$ by recoloring
			any edges of colors $5$ and $6$ by colors $1$ and $4$, respectively.
			It is straightforward to verify that the colorings
			$\varphi$ and $\alpha'$ taken together form a $2$-improper
			interval coloring of $G$.
		\end{proof}

		Using the same vertex splitting technique as several times before,
		we deduce the following corollary.  Note that if $G$ is outerplanar
		and $v \in V(G)$, then given integers $k$ and $l$ such that
		$k+l=d_G(v)$,
		it is always possible to split the vertex $v$
		into two new vertices $v'$ and $v''$ of degrees $k$ and $l$,
		respectively, so that the resulting graph
		is outerplanar (or a union of vertex-disjoint outerplanar graphs). 
		We state this observation as a lemma.
		
		\begin{lemma}
		\label{lem:split}
			If $G$ is outerplanar, $v$ is a vertex of $G$ and $k$ and $l$
			are positive integers satisfying $d_G(v) = k+l$, then
			we can split the vertex $v$ into two new vertices of degrees
			$k$ and $l$, respectively, in such a way that the resulting graph
			is outerplanar (or a union of vertex-disjoint outerplanar graphs).
		\end{lemma}

		\begin{corollary}
			If $G$ is an outerplanar graph, then
			$\muint(G) \leq \left\lceil
			\frac{\Delta(G)}{4} \right\rceil +1$.
		\end{corollary}
		\begin{proof}
			By Proposition \ref{prop:outerplanar}, we may assume that
			$\Delta(G) \geq 9$.
			As in the preceding proof, it suffices to consider the case when
			$G$ is $2$-connected. Let $C$ be a Hamiltonian cycle of $G$;
			we color $C$ as in the proof of Proposition \ref{prop:outerplanar}.
			The result now follows by splitting all vertices of $G-E(C)$
			into as many vertices of degree
			$4$ as possible, and possibly one additional vertex of degree at most
			$3$; 
			 by repeatedly applying Lemma \ref{lem:split}, this can be done
			so that the resulting graph $J$ is outerplanar (or a union
			of disjoint outerplanar graphs).

			Now, since $\Delta(J) =4$, $J$ has a proper $4$-edge coloring. 
			This proper edge coloring, together with the coloring of $C$
			is the required improper interval edge coloring of $G$.
		\end{proof}
		\bigskip

\subsection{Complete multipartite graphs}

In this section we prove an upper bound for the impropriety
of complete multipartite graphs.
		A graph $G$ is called {\em complete $r$-partite} ($r\geq 2$) if its vertices can be partitioned into $r$ nonempty independent sets $V_1,\ldots,V_r$ such that each vertex in $V_i$ is adjacent to all the other vertices in $V_j$ for $1\leq i<j\leq r$. Let $K_{n_{1},n_{2},\ldots,n_{r}}$ denote a complete $r$-partite graph
        with independent sets $V_1,V_2,\ldots,V_r$ of sizes $n_{1},n_{2},\ldots,n_{r}$.

\begin{theorem}
\label{completemulti} For any $n_{1},n_{2},\ldots,n_{r}\in
\mathbb{N}$, $$\muint\left(K_{n_{1},n_{2},\ldots,n_{r}}\right)\leq \left\lceil\frac{r}{2}\right\rceil.$$
\end{theorem} 

\begin{proof} Without loss of generality, we may assume that $n_{1}\geq n_{2}\geq \cdots\geq n_{r}$. We partition the independent sets 
$V_1,\dots, V_r$
into two groups: $X_{1},X_{2},\ldots,X_{\left\lceil\frac{r}{2}\right\rceil}$ (first $\left\lceil\frac{r}{2}\right\rceil$ independent sets) of sizes $n_{1},n_{2},\ldots,n_{\left\lceil\frac{r}{2}\right\rceil}$ and 
$Y_{1},Y_{2},\ldots,Y_{\left\lfloor\frac{r}{2}\right\rfloor}$
of sizes $n_{\left\lceil\frac{r}{2}\right\rceil+1},n_{\left\lceil\frac{r}{2}\right\rceil+2},\ldots,n_{r}$ 
(remaining $\left\lfloor\frac{r}{2}\right\rfloor$ independent sets). Let
$X=X_{1}\cup X_{2}\cup\cdots \cup X_{\left\lceil\frac{r}{2}\right\rceil}$ and 
$Y=Y_{1}\cup Y_{2}\cup\cdots \cup Y_{\left\lfloor\frac{r}{2}\right\rfloor}$. We also label 
the vertices of $X$ and $Y$ as follows: $$X=\left\{x_{1},x_{2},\ldots,x_{n_{1}},x_{n_{1}+1},\ldots,x_{n_{1}+n_{2}},\ldots,x_{n_{1}+n_{2}+\cdots+n_{\left\lceil\frac{r}{2}\right\rceil}}\right\}$$ 
and 
$$Y=\left\{y_{1},y_{2},\ldots,y_{n_{\left\lceil\frac{r}{2}\right\rceil+1}},y_{n_{\left\lceil\frac{r}{2}\right\rceil+1}+1},\ldots,y_{n_{\left\lceil\frac{r}{2}\right\rceil+1}+n_{\left\lceil\frac{r}{2}\right\rceil+2}},\ldots,y_{n_{\left\lceil\frac{r}{2}\right\rceil+1}+n_{\left\lceil\frac{r}{2}\right\rceil+2}+\cdots+n_{r}}\right\}.$$ 

Let $s_{i}=\sum_{j=1}^{i}n_{j}$ ($1\leq i\leq \left\lceil\frac{r}{2}\right\rceil$) and $t_{i}=\sum_{j=\lceil\frac{r}{2}\rceil+1}^{\lceil\frac{r}{2}\rceil+i}n_{j}$ ($1\leq i\leq \left\lfloor\frac{r}{2}\right\rfloor$).

Now let us consider the subgraphs $H$ and $H'$ of $K_{n_{1},n_{2},\ldots,n_{r}}$ induced by the vertices of $X$ and $Y$, respectively. We first define an edge coloring $\alpha$ of $H \cup H'$.


\begin{description}
\item[($1$)] For $1\leq i\leq \lceil\frac{r}{2}\rceil-1$, $1\leq j\leq s_{i}$ and $1\leq k\leq n_{i+1}$, let
$$\alpha \left(x_{j}x_{s_{i}+k}\right)=j+k-1.$$

\item[($2$)] For $1\leq i\leq \lfloor\frac{r}{2}\rfloor-1$, $1\leq j\leq t_{i}$ and $1\leq k\leq n_{\lceil\frac{r}{2}\rceil+1+i}$, let
$$\alpha \left(y_{j}y_{t_{i}+k}\right)=j+k-1.$$
\end{description}

By the definition of $\alpha$, we have 

\begin{description}
\item[(a)] for $1\leq k\leq n_{1}$, 
$$S_{H}(x_{k},\alpha)=[k,k+n_{2}-1],$$

\item[(b)] for $1\leq i\leq \lceil\frac{r}{2}\rceil-2$ and $1\leq k\leq n_{i+1}$,
$$S_{H}(x_{s_{i}+k},\alpha)=[k,s_{i}+k-1]\cup [s_{i}+k,s_{i+1}+k-1]=[k,s_{i+1}+k-1],$$

\item[(c)] for $1\leq k\leq n_{\lceil\frac{r}{2}\rceil}$, 
$$S_{H}\left(x_{s_{\lceil\frac{r}{2}\rceil-1}+k},\alpha\right)=\left[k,s_{\lceil\frac{r}{2}\rceil-1}+k-1\right],$$

\item[(d)] for $1\leq k\leq n_{\lceil\frac{r}{2}\rceil+1}$, 
$$S_{H'}(y_{k},\alpha)=\left[k,k+n_{\lceil\frac{r}{2}\rceil+2}-1\right],$$

\item[(e)] for $1\leq i\leq \lfloor\frac{r}{2}\rfloor-2$ and $1\leq k\leq n_{\lceil\frac{r}{2}\rceil+1+i}$,
$$S_{H'}(y_{t_{i}+k},\alpha)=[k,t_{i}+k-1]\cup [t_{i}+k,t_{i+1}+k-1]=[k,t_{i+1}+k-1],$$

\item[(f)] for $1\leq k\leq n_{r}$, 
$$S_{H'}\left(y_{t_{\lfloor\frac{r}{2}\rfloor-1}+k},\alpha\right)=\left[k,t_{\lfloor\frac{r}{2}\rfloor-1}+k-1\right].$$
\end{description}

Note that for every vertex $v$ of $H \cup H'$, each color
can occur at most $\lceil\frac{r}{2}\rceil-1$ times at $v$ 
under the coloring $\alpha$.
Hence, $\alpha$ is an $(\lceil\frac{r}{2}\rceil-1)$-improper interval 
coloring of $H \cup H'$.  

Next, we define an edge coloring $\beta$ of $K_{n_{1},n_{2},\ldots,n_{r}}-E(H\cup H')$ as follows: for $1\leq i\leq s_{\lceil\frac{r}{2}\rceil}$ and $1\leq j\leq t_{\lfloor\frac{r}{2}\rfloor}$, let
$$\beta \left(x_{i}y_{j}\right)=i+j-1.$$

Now we are able to define an edge coloring $\gamma$ of $K_{n_{1},n_{2},\ldots,n_{r}}$ 
by taking the colorings $\alpha$ and $\beta$ together; that is,
for any $e\in E(K_{n_{1},n_{2},\ldots,n_{r}})$, we set

$$\gamma(e) =
			\begin{cases}
				\alpha(e), \text{ if $e\in E(H\cup H')$}, \\
				\beta(e), \text{ otherwise.}
			\end{cases}
$$

By the definition of $\gamma$, we have 

\begin{description}
\item[(a')] for  $1\leq k\leq n_{1}$, 
$$S(x_{k},\gamma)=[k,k+n_{2}-1]\cup \left[k,k+t_{\lfloor\frac{r}{2}\rfloor}-1\right],$$

\item[(b')] for $1\leq i\leq \lceil\frac{r}{2}\rceil-2$ and $1\leq k\leq n_{i+1}$,
$$S(x_{s_{i}+k},\gamma)=[k,s_{i+1}+k-1]\cup \left[s_{i}+k,s_{i}+k+t_{\lfloor\frac{r}{2}\rfloor}-1\right],$$

\item[(c')] for $1\leq k\leq n_{\lceil\frac{r}{2}\rceil}$, 
$$S\left(x_{s_{\lceil\frac{r}{2}\rceil-1}+k},\gamma\right)=[k,s_{\lceil\frac{r}{2}\rceil-1}+k-1]\cup \left[s_{\lceil\frac{r}{2}\rceil-1}+k,s_{\lceil\frac{r}{2}\rceil-1}+k+t_{\lfloor\frac{r}{2}\rfloor}-1\right],$$

\item[(d')] for $1\leq k\leq n_{\lceil\frac{r}{2}\rceil+1}$, 
$$S(y_{k},\gamma)=\left[k,k+n_{\lceil\frac{r}{2}\rceil+2}-1\right]\cup \left[k,k+s_{\lceil\frac{r}{2}\rceil}-1\right],$$

\item[(e')] for $1\leq i\leq \lfloor\frac{r}{2}\rfloor-2$ and $1\leq k\leq n_{\lceil\frac{r}{2}\rceil+1+i}$,
$$S(y_{t_{i}+k},\gamma)=[k,t_{i+1}+k-1]\cup \left[t_{i}+k,t_{i}+k+s_{\lceil\frac{r}{2}\rceil}-1\right],$$

\item[(f')] for $1\leq k\leq n_{r}$, 
$$S\left(y_{t_{\lfloor\frac{r}{2}\rfloor-1}+k},\gamma\right)=\left[k,t_{\lfloor\frac{r}{2}\rfloor-1}+k-1\right]\cup \left[t_{\lfloor\frac{r}{2}\rfloor-1}+k,t_{\lfloor\frac{r}{2}\rfloor-1}+k+s_{\lceil\frac{r}{2}\rceil}-1\right].$$
\end{description}

It is not difficult to see that $\gamma$ is an $\lceil\frac{r}{2}\rceil$-improper interval coloring of $K_{n_{1},n_{2},\ldots,n_{r}}$; thus $\muint\left(K_{n_{1},n_{2},\ldots,n_{r}}\right)\leq \left\lceil\frac{r}{2}\right\rceil$.
\end{proof}

\begin{corollary}
			If $G$ is a complete $3$-partite or $4$-partite graph, then
			$\muint(G) \leq 2$.
\end{corollary}

In fact, we
believe that a more general result is true:

\begin{conjecture}
\label{conj:partite}
			If $G$ is a complete $r$-partite graph, then
			$\muint(G) \leq 2$.
\end{conjecture}

Since there are complete multipartite graphs of Class 2,
Conjecture \ref{conj:partite}, if true, would be best possible.

	\subsection{Cartesian products of graphs}

		In this section we consider the impropriety of Cartesian products of graphs. 
				The Cartesian product $G\square
H$ of two graphs $G$ and $H$ is defined by setting
$$V(G\square H)=V(G)\times V(H), \text{and}$$
$$E(G\square H)=\{(u_{1},v_{1})(u_{2},v_{2})\colon\,
(u_{1}=u_{2}\wedge v_{1}v_{2}\in E(H))\vee (v_{1}=v_{2}\wedge
u_{1}u_{2}\in E(G))\}.$$

\begin{proposition}
\label{cartesianproduct} For any graphs $G$ and $H$, $$\muint(G\square H)\leq \max\left\{\muint(G),\muint(H)\right\}.$$
\end{proposition} 
\begin{proof} In the proof of this theorem we follow the idea from \cite{GiaroKubale2} (Theorem 2.4). Let $\alpha$ be a $k_{1}$-improper interval $t_{1}$-coloring of $G$ and $\beta$ be a $k_{2}$-improper interval $t_{2}$-coloring of $H$, where $k_{1}=\muint(G)$ and $k_{2}=\muint(H)$.  

We define an edge coloring $\gamma$ of $G\square H$ as follows: for every $(u_{1},v_{1})(u_{2},v_{2})\in E(G\square H)$, let

$$\gamma\left((u_{1},v_{1})(u_{2},v_{2})\right)=
			\begin{cases}
				\alpha(u_{1}u_{2})+\underline{S}(v_{1},\beta)-1, & \text{if $v_{1}=v_{2}$ and $u_{1}u_{2}\in E(G)$},\\
				\beta(v_{1}v_{2})+\overline{S}(u_{1},\alpha), & \text{if $u_{1}=u_{2}$ and $v_{1}v_{2}\in E(H)$.}\\
			\end{cases}
$$

By the definition of $\gamma$, for every vertex $(u,v)\in V(G\square H)$, we have 

$$S((u,v),\gamma) = \left[\underline{S}(u,\alpha)+\underline{S}(v,\beta)-1,\overline{S}(u,\alpha)+\underline{S}(v,\beta)-1\right]\cup$$
$$\cup\left[\underline{S}(v,\beta)+\overline{S}(u,\alpha),\overline{S}(v,\beta)+\overline{S}(u,\alpha)\right]=\left[\underline{S}(u,\alpha)+\underline{S}(v,\beta)-1,\overline{S}(u,\alpha)+\overline{S}(v,\beta)\right].$$

Since for every vertex $(u,v)$ of $G\square H$, each color
can occur at most $\max\{k_{1},k_{2}\}$ times at $(u,v)$ 
under the coloring $\gamma$, this implies that $\gamma$ is a $\max\{k_{1},k_{2}\}$-improper interval $(t_{1}+t_{2})$-coloring of $G\square H$. 
Thus, $\muint(G\square H)\leq \max\left\{\muint(G),\muint(H)\right\}$.
\end{proof}

Clearly, this upper bound on the impropriety in Theorem \ref{cartesianproduct} is sharp for all Cartesian products of graphs when the factors are interval colorable. Let us note that there are graphs $G$ and $H$ such that $\muint(G\square H)< \max\left\{\muint(G),\muint(H)\right\}$. For example, 
if $G$ and $H$ are  both isomorphic to the Petersen graph, then, by Proposition \ref{prop:WheelReg}, $\muint(G)=\muint(H)=2$, but 
$\muint(G\square H)= 1$, since $G$ and $H$ contain perfect matchings \cite{PetDMGT}.
On the other hand, if we consider the Cartesian product of two odd cycles $C_{2m+1}\square C_{2n+1}$, then again, by Proposition \ref{prop:WheelReg}, $\muint(C_{2m+1})=\muint(C_{2n+1})=2$, but in this case $\muint(C_{2m+1}\square C_{2n+1})=2$, since $C_{2m+1}\square C_{2n+1}$ is Class 2. So, the upper bound on the impropriety in Theorem \ref{cartesianproduct} is also sharp for all Cartesian products of regular graphs when the factors and the Cartesian product of factors are Class 2.

		\section{The number of colors in an improper interval coloring}
		
Following \cite{Hudak}, we denote by $\hat t (G)$ the maximum
number of colors used in an improper interval edge coloring of $G$.
In \cite{Hudak}, the authors proved the following two results.

\begin{theorem} \cite{Hudak}
\label{triangle-free} For each connected triangle-free graph $G$ on $n$ vertices, 
$\hat{t}(G)\leq n-1$. Moreover, the upper bound is sharp.
\end{theorem}

\begin{theorem} \cite{Hudak}
\label{generalbound_old} For each connected graph $G$ on $n$ vertices, 
$\hat{t}(G)\leq 2n-1$.
\end{theorem}

Here we slightly improve the general upper bound from the last theorem.

\begin{theorem}
\label{generalbound_new} For each connected graph $G$ on $n$ vertices ($n\geq 2$), $\hat{t}(G)\leq 2n-3$.
\end{theorem}
\begin{proof}
Let $V(G)=\{v_{1},v_{2},\ldots,v_{n}\}$ and $\alpha$ be an improper interval
$\hat{t}(G)$-coloring of $G$.
 Define an auxiliary graph $H$ as follows:
\begin{center}
$V(H)=U\cup W$, where
\end{center}
\begin{center}
$U=\{u_{1},u_{2},\ldots,u_{n}\}$, $W=\{w_{1},w_{2},\ldots,w_{n}\}$
and
\end{center}
\begin{center}
$E(H)=\left\{u_{i}w_{j},u_{j}w_{i}\colon\,v_{i}v_{j}\in E(G), 1\leq i\leq n,1\leq j\leq n\right\}\cup \{u_{i}w_{i}\colon\,1\leq i\leq n\}$.
\end{center}

Clearly, $H$ is a bipartite graph with $\vert V(H)\vert = 2n$.\\

Define an edge-coloring $\beta$ of $H$ as follows:
\begin{description}
\item[(1)] for every edge $v_{i}v_{j}\in E(G)$, let $\beta (u_{i}w_{j})=\beta(u_{j}w_{i})=\alpha(v_{i}v_{j})+1$,

\item[(2)] for $i=1,2,\ldots,n$, let $\beta(u_{i}w_{i})=\overline S(v_{i},\alpha)+2$.
\end{description}

It is easy to see that $\beta$ is an edge-coloring of the graph $H$ with colors $2,3,\ldots,\hat{t}(G)+2$ and $\underline S(u_{i},\beta)=\underline
S(w_{i},\beta)$ for $i=1,2,\ldots,n$. We 
construct an improper interval
$(\hat{t}(G)+2)$-coloring of the graph $H$ by picking an edge
$u_{i_{0}}w_{i_{0}}$ with 
$\underline S\left(u_{i_{0}},\beta\right)=
\underline S\left(w_{i_{0}},\beta\right)=2$ and recoloring it 
with color $1$. The obtained coloring is an improper interval $(\hat{t}(G)+2)$-coloring of $H$. 
Since $H$ is a connected bipartite graph, by Theorem
\ref{triangle-free}, we have
\begin{center}
$\hat{t}(G)+2\leq \vert V(H)\vert -1 = 2n-1$,
\end{center}
thus
$\hat{t}(G)\leq 2n-3$.
\end{proof}

We note that the upper bound in the preceding theorem is sharp
by the example of a complete graph with only two vertices.


\begin{thebibliography}{99}

\bibitem{AndrewsJacobson}
J. Andrews, M. S. Jacobson. On a generalization of chromatic number,
Congressus Numerantium, 47:33--48, 1985.


\bibitem{AsrDenHag} A.S. Asratian, T.M.J. Denley, R. Haggkvist, Bipartite Graphs and their Applications, Cambridge University Press, Cambridge, 1998.

\bibitem{AsrKam} A.S. Asratian, R.R. Kamalian, Interval colorings of edges of a
multigraph, Appl. Math. 5 (1987) 25--34 (in Russian).

\bibitem{AsrKamJCTB} A.S. Asratian, R.R. Kamalian, Investigation on interval
edge-colorings of graphs, J. Combin. Theory Ser. B 62 (1994) 34--43.

\bibitem{Axen} M.A. Axenovich, On interval colorings of planar graphs, Congr. Numer. 159 (2002) 77-94.

\bibitem{CarlJToft} C.J. Casselgren, B. Toft, On interval edge colorings of biregular bipartite graphs with small vertex degrees, J. Graph Theory 80 (2015) 83--97.


\bibitem{CowenGoddardJesurum}
L.J. Cowen, W. Goddard, C.E. Jesurum, Coloring with defect. SODA '97 	
Proceedings of the eighth annual ACM-SIAM symposium on discrete algorithms, 
548--557.

\bibitem{FengHuang} Y. Feng, Q. Huang, Consecutive edge-coloring of the generalized
$\theta$-graph, Discrete Appl. Math. 155 (2007) 2321--2327.

\bibitem{Fiorini}
S. Fiorini, The chromatic index of outerplanar graphs, J. Combin. Theory Ser. B 18 (1975), 35--38.

\bibitem{Fournier} J.C. Fournier, Coloration des aretes dun graphe, 
Cahiers du CERO (Bruxelles) 15 (1973) 311--314.

\bibitem{GiaroKubale1} K. Giaro, M. Kubale, Consecutive edge-colorings of complete and incomplete Cartesian products of graphs, Congr. Numer. 128 (1997) 143--149.

\bibitem{GiaroKubale2} K. Giaro, M. Kubale, Compact scheduling of zero-one time operations in multi-stage systems, Discrete Appl. Math. 145 (2004) 95--103.

\bibitem{GiaroKubaleMalaf1} K. Giaro, M. Kubale, M. Ma\l afiejski, On the deficiency of bipartite graphs, Discrete Appl. Math. 94 (1999) 193--203.

\bibitem{GiaroKubaleMalaf2} K. Giaro, M. Kubale, M. Ma\l afiejski, Consecutive colorings of the edges of general graphs, Discrete Math. 236 (2001) 131--143.

\bibitem{Hansen} H.M. Hansen, Scheduling with minimum waiting periods, MSc Thesis, Odense University, Odense, Denmark, 1992 (in Danish).

\bibitem{HansonLotenToft} D. Hanson, C.O.M. Loten, B. Toft, On interval colorings of bi-regular bipartite graphs, Ars Combin. 50 (1998) 23--32.

\bibitem{HararyJones} F. Harary, K. Jones, Conditional colorability II: Bipartite variations, Congressus Numerantium 50 (1985) 205--218.

\bibitem{Hudak}
P. Hudak, F. Kardos, T. Madaras, M. Vrbjarova,
On improper interval edge colourings, Czechoslovak Mathematical Journal 66 (2016), 1119--1128.

\bibitem{JensenToft} T.R. Jensen, B. Toft, Graph Coloring problems, Wiley Interscience, 1995.

\bibitem{Kampreprint} R.R. Kamalian, Interval colorings of complete bipartite graphs and trees, preprint, Comp. Cen. of Acad. Sci. of Armenian SSR, 
Yerevan, 1989 (in Russian).

\bibitem{KamDiss} R.R. Kamalian, Interval edge colorings of graphs, Doctoral Thesis, Novosibirsk, 1990.

\bibitem{KamMir} R.R. Kamalian, A.N. Mirumian, Interval edge colorings of bipartite graphs of some class, Dokl. NAN RA 97 (1997) 3--5 (in Russian).

\bibitem{PetDMGT} P.A. Petrosyan, Interval edge colorings of some products of graphs, Discuss. Math. Graph Theory 31 (2011) 357-–373.

\bibitem{PetrosKhacha} P.A. Petrosyan, H.H. Khachatrian, Interval non-edge-colorable bipartite graphs and multigraphs, J. Graph Theory 76 (2014) 200--216.

\bibitem{Seva} S.V. Sevast'janov, Interval colorability of the edges of a
bipartite graph, Metody Diskret. Analiza 50 (1990) 61--72 (in
Russian).

\bibitem{Wood}
David R. Wood,
Defective and Clustered Graph Colouring, preprint
available on Arxiv.

\end{thebibliography}
\end{document}